\documentclass[12pt]{article}

\usepackage{tikz}
\usepackage{subfigure}
\usepackage[english]{babel}
\usepackage{amsfonts,amssymb,amsmath,latexsym,amsthm}
\usepackage{multirow}

\newtheorem{thm}{Theorem}[section]

\newtheorem{lem}[thm]{Lemma}
\newtheorem{prop}[thm]{Proposition}

\newtheorem{claim}{Claim}

\begin{document}

\title{Minimum degree of 3-graphs without long linear paths
\thanks{The work was supported by NNSF of China (No. 11671376),  NSF of Anhui Province (No. 1708085MA18), and Anhui Initiative in Quantum Information Technologies (AHY150200).}
}
\author{Yue Ma$^a$, \quad Xinmin Hou$^b$,\quad Jun Gao$^c$\\
\small $^{a,b,c}$ Key Laboratory of Wu Wen-Tsun Mathematics\\
\small School of Mathematical Sciences\\
\small University of Science and Technology of China\\
\small Hefei, Anhui 230026, China.
}

\date{}

\maketitle

\begin{abstract}
A well known theorem in graph theory states that every graph $G$ on $n$ vertices and minimum degree at least $d$ contains a path of length at least $d$, and if  $G$ is connected and  $n\ge 2d+1$ then $G$ contains a path of length at least $2d$ (Dirac, 1952). In this article, we give an extension of Dirac's result to hypergraphs. We determine asymptotic lower bounds of the minimum degrees of 3-graphs to guarantee linear paths of specific lengths, and the lower bounds are tight up to a constant.
\end{abstract}

\section{Introduction}
An {\it $r$-uniform hypergraph (or $r$-graph for short)} is a pair $H=(V, E)$, where $V$ is a set of elements called vertices, and $E$ is a collection of subsets of $V$ with uniform size $r$ called edges.
In this article, all $r$-graphs $H$ considered are simple, i.e. $H$ contains no multiple edges.
We call $|V|$ {\it the order} of $H$ and $|E|$ {\it the size} of $H$, also denoted by $|H|$ or $e(H)$.  We write graph for $2$-graph for short.
A {\it linear $k$-path} (or a {\it linear path of length $k$}), denoted by $P_k$, is a collection of $k$ edges $\{e_1,e_2,...,e_k\}$ such that $|e_i\cap e_j|=1$ if $|i-j|=1$ and $e_i\cap e_j=\emptyset$ otherwise. Given $S\subseteq V(H)$, the {\it degree} of $S$, denote by $d_H(S)$, is the number of edges of $H$ containing $S$. The minimum $s$-degree $\delta_s(H)$ of $H$ is the minimum of $d_H(S)$ over all $S\subseteq V(H)$ of size $s$.
We call $\delta_1(H)$ the \emph{minimum degree} of $H$, that is $\delta_1(H)=\min\{d_H(v) : v\in V(H)\}$. Let $N_H(S)=\{ T : S\cup T\in E(H)\}$.
Given two $r$-graphs $H$ and $F$, we say $H$ is $F$-free if $H$ contains no subgraph isomorphic to $F$. Given two integers $a,b$ with $a<b$, write $[a,b]$ for the set $\{a, a+1, \ldots, b\}$.




The following results are well known in graph theory related to minimum degree and the lengths of paths in a graph, two of them were due to Dirac.
Note that, for a graph $G$, we write a path for a linear path and $\delta(G)$ for $\delta_1(G)$.
\begin{thm}\label{THM: Dirac52}
Let $G$ be a graph on $n$ vertices with minimum degree $d$.

(i) $G$ contains a path of length at least $d$.

(ii) (Dirac 1952, Theorem 3 in~\cite{Dirac52}) If $G$ is connected and  $n\le 2d$ then $G$ contains a path of length at least $n$ (i.e. a Hamiltonian path).

(iii) (Dirac 1952, Theorem 4 in~\cite{Dirac52}) If $G$ is connected and  $n\ge 2d+1$ then $G$ contains a path of length at least $2d$.
\end{thm}

The famous Erd\H{o}s-Gallai Theorem~\cite{Erdos-Gallai59} states that every graph on $n$ vertices and $\frac{(k-1)n}2$ edges contains a path of length $k$, this can be viewed as an average degree version of Theorem~\ref{THM: Dirac52}. Erd\H{o}s-Gallai Thoerem was improved later by Faudree and Schelp~\cite{Faudree-Schelp75}, and the connected version was given by Balister, Gy\H{o}ri, Lehel, and Schelp~\cite{BGLS08} in 2008. An hypergraph extension of Erd\H{o}s-Gallai Theorem was solved by Gy\H{o}ri, Katona, and Lemons~\cite{GKL16} and Davoodi, Gy\H{o}ri, Methuku, and Tompkins~\cite{DGMT-18}. Motivated by these results, in this article, we give a hypergraph  version of Theorem~\ref{THM: Dirac52}.
\begin{thm}\label{THM: Main1}
Let $k$ be a nonnegative integer.

(1) Every 3-graph $H$ on $n\ge 4k+19$ vertices with $\delta_1(H)\ge kn+6k^2-3k+3$ contains a linear path of length $2k+1$.

(2) Every 3-graph $H$ on $n\ge 4k+21$ vertices with $\delta_1(H)\ge kn+6k^2+7k+6$ contains a linear path of length $2k+2$.

\end{thm}

The lower bound is tight up to an error term $O(k^2)$. To verify this, let $S_r(n,k)$ be the $r$-graph on vertex set $A\cup B$ with $|A|=k$ and $|B|=n-k$, and edge set $$E=\{e : e\subset A\cup B \mbox{ with $|e|=r$ and } e\cap A\not=\emptyset\};$$
let $C_r(n, s)$ be the $r$-graph with vertex set $S\cup T$ with $|S|=s$ and $|T|=n-s$, and edge set $$E=\{e : e\subset S\cup T \mbox{ with $|e|=r$ and } S\subset e\};$$
$S^+_r(n,k)$ is the $r$-graph obtained from $S_r(n,k)$ by embedding a copy of $C_r(n-k, 2)$ in $B$. The $r$-graphs $S_r(n,k)$ and $S^+_r(n,k)$ have also been defined by Kostochka, Mubayi, and Verstra\"{e}te~\cite{KMV_15} using a different notation.
The following proposition can be checked directly from the definitions of $S_r(n,k)$ and $S_r^+(n,k)$.

\begin{prop}\label{PROP: EX}
Let $k\ge 1$ be an integer.

(1) $S_3(n,k)$ is $P_{2k+1}$-free and $\delta_1(S_3(n,k))=kn-\frac{k^2}2-\frac{3k}2$;

(2) $S_3^+(n,k)$ is $P_{2k+2}$-free and $\delta_1(S_3^+(n,k))=kn-\frac{k^2}2-\frac{3k}2+1$.

\end{prop}
Proposition~\ref{PROP: EX} shows that the lower bound given in Theorem~\ref{THM: Main1} is tight up to a constant $c(k)$ depending on $k$. In fact, we believe that $S_3(n,k)$ and $S^+_3(n,k)$ are extremal graphs for $P_{2k+1}$-free and $P_{2k+2}$-free graphs with maximum minimum degree, respectively. We leave this as an open question.

The rest of the article is arranged as follows.
In Section 2, we give the proof of Theorem~\ref{THM: Main1}. We give some discussions and remarks in Section 3.

\section{Proof of Theorem~\ref{THM: Main1}}
For the special case $k=0$, we have better lower bounds than the ones given in Theorem~\ref{THM: Main1}.
\begin{lem}\label{LEM: k=0}
(1) Every 3-graph $H$ on $n\ge 3$ vertices with $\delta_1(H)\ge 1$ contains a linear path of length $1$.

(2) Every 3-graph $H$ on $n\ge 5$ vertices with $\delta_1(H)\ge 4$ contains a linear path of length $2$.

\end{lem}
\begin{proof}
(1) It is trivial since every edge $e\in E(H)$ is a $P_1$ in $H$.

(2) Choose an edge $e_1=\{a,b,c_1\}\in E(H)$. Since $d_H(a)\ge \delta_1(H)\ge 4$, we can pick three distinct edges $e_2,e_3,e_4\in E(H)\setminus\{e_1\}$. If there exist $i,j\in [1,4]$ such that $|e_i\cap e_j|=1$, then $\{e_i,e_j\}$ induces a $P_2$ in $H$. So we assume that $|e_i\cap e_j|=2$ for all $i,j\in [1,4]$.  
Let $e_2=\{a,b,c_2\}$, where $c_2\neq c_1$.

 If $b\in e_3$ or $b\in e_4$, without loss of generality, assume $e_3=\{a,b,c_3\}$, where $c_1,c_2, c_3$ are pairwise distinct. Now, consider $c_1$, there must exist an edge $e'$ with $c_1\in e'$ and  $e'\not=e_1$. Clearly, $|e'\cap e_i|=2$ (otherwise, we have a $P_2$ in $H$) for $i=1,2,3$. Without loss of generality, assume $e'=\{a,c_1,d\}$. Then $d\not=b$.      So, at least one of $c_2, c_3$ is different from $d$, which contradicts to $|e'\cap e_i|=2$, $i=2,3$.

Now assume $b\not\in e_3$ and $b\notin e_4$. Since $|e_1\cap e_3|=2$ and $|e_2\cap e_3|=2$, we have $c_1, c_2\in e_3$, which means $e_3=\{a,c_1,c_2\}$. With the same reason, we have $e_4=\{a,c_1,c_2\}=e_3$, a contradiction.
\end{proof}

\noindent{\bf Remark.} $n\ge 5$ and $\delta_1(H)\ge 4$ is best possible. For example, the complete 3-graph $K_4^3$ has minimum degree 3 but does not contain a linear path of length two.

For a 3-graph $H$, we write $xyz\in E(H)$ for $\{x,y,z\}\in E(H)$, write $P_t=(x_0,x_1,...,x_{2t})$ for the linear path $P_t=\{x_0x_1x_2, x_2x_3x_4,...,x_{2t-2}x_{2t-1}x_{2t}\}$ in $H$, and for distinct $a,b\in\{0, 1,2, \ldots, 2t \}$, define $d_{P}(a,b):=|N_{H}(\{x_a,x_b\})\setminus V(P)|$.
A {\it linear $k$-cycle} in an $r$-graph, denoted by $C_k$, is a collection of $k$ edges $\{e_1,e_2,...,e_k\}$ such that $|e_i\cap e_j|=1$ if $|i-j|=1$ or $k-1$ and $e_i\cap e_j=\emptyset$ otherwise. Let $C_k^{+} $ be the $r$-graph, called a \emph{$k$-cycle with a parallel edge}, obtained from $C_k$ by adding a new vertex $v$ an edge $f$ with the property that $v\in f$ and there is an edge $e\in E(C_k)$  such that $(C_k-e)\cup\{f\}$ is also a linear cycle of length $k$ and $|f\cap e|=2$.
Define \begin{equation*}
g(n,t)=\left\{\begin{array}{ll}
         \frac{t-1}2n+\frac 32t^3-\frac 92t+6 & \mbox{ $t$ is odd,} \\
         \frac{t-2}2n+\frac 32t^3-\frac {5}2t+6 & \mbox{ $t$ is even.}
       \end{array}
       \right.
       \end{equation*}

\begin{lem}\label{LEM: Cycle}
Given integers $t\ge 3$ and $n\ge 2t+17$, let $H$ be a 3-graph on $n$ vertices and $\delta_1(H)\ge g(n,t)$. If $H$ is $P_{t+1}$-free then $H$ contains no $C_{t+1}^+$ as a subgraph.
\end{lem}
\begin{proof}
Suppose to the contrary that there exists a $C_{t+1}^+$ in $H$. Write $$C_{t+1}^+=(x_0,...,x_{2t})\cup\{x_0x_{2t+1}x_{2t}\}\cup\{x_0vx_{2t}\}.$$ Let $X=\{x_0,x_1,...,x_{2t},x_{2t+1}\}$.
Since $n\ge 2t+17$ and $t\ge 3$, we have $\delta_1(H)\ge g(n,t)>\binom{2t+2}{2}$.
\begin{claim}\label{CLAIM:c1}
There must exist an element  $T\in N_H(v)$ such that $|T\cap X|=1$.
\end{claim}
 We first claim that there is no $T\in N_H(v)$ such that $T\cap X=\emptyset$.
If not, assume that there is a $T$ with $|T\cap X|=0$. Let $e=\{v\}\cup T=\{v,v_1,v_2\}$. Then $(v_2,v_1,v,x_{2t},x_0,x_1,x_2,x_3,...,x_{2t-2})$ is a linear path of length $t+1$ in $H$, a contradiction. 
Now assume $|T\cap X|=2$ for any $T\in N_H(v)$.
Then
$d_H(v)\le \binom{2t+2}{2}<\delta_1(H)$, a contradiction. This completes the proof of the claim.

By Claim~\ref{CLAIM:c1}, we can choose a $T\in N_H(v)$ with $|T\cap X|=1$.
Let $T\cup\{v\}=\{v,v',x_i\}$, where $v'\notin X$ and $i\in [0, 2t+1]$.
If $i$ is even, then $$(v,v',x_i,x_{i+1},x_{i+2},...,x_{2t},x_{2t+1},x_0,x_1,...,x_{i-2})$$ is a linear path of length $t+1$ in $H$, a contradiction.
Now assume $i$ is odd. then $$(v,v',x_i,x_{i-1},x_{i+1},x_{i+2},...,x_{2t},x_{2t+1},x_0,x_1,...,x_{i-3})$$ is a linear path of length $t+1$ in $H$, a contradiction, too.
\end{proof}


\begin{lem}\label{PROP: 4}
Given integers $t\ge 3$ and $n\ge 2t+17$, let $H$ be a 3-graph on $n$ vertices and $\delta_1(H)\ge g(n,t)$. If $H$ is $P_{t+1}$-free and $P=(x_0,x_1,...,x_{2t})$ is a linear path of length $t$ in $H$, then the following statements hold:

(i) $d_P(0,2t)\le 1$;

(ii) if there is some $k\in [0,t-1]$ such that $d_P(a,2k+1)>0$ for $a\in \{0, 2t\}$, then $d_P(0, 2k+1)+d_P(2t,2k+1)\le 2$;


(iii) if there is some $k\in [0,t-1]$ such that $d_P(0,2k+2)>0$ and $d_P(2t,2k)>0$, then $d_P(0,2k+2)+d_P(2t,2k)\le 4$;

(iv) if there is some $k\in [0,t-1]$ and some $\ell\in\{0, t\}$ such that $d_P(2k,2k+2)>0$ and $d_P(2\ell,2k+1)>0$, then $d_P(2k,2k+2)+d_P(2\ell,2k+1)\le 2$.


\end{lem}
\begin{proof} By Lemma~\ref{LEM: Cycle}, $H$ is $C_{t+1}^+$-free.

(i) It is a direct corollary of Lemma~\ref{LEM: Cycle}.

(ii) Let $d_P(0, 2k+1)=i$ and $d_P(2t,2k+1)=j$. Then $i,j\ge 1$. Without loss of generality, assume $i\ge j$.  If $i+j\ge 3$ then $i\ge 2$. So there must exist $y\in N_H(\{x_0, x_{2k+1}\})\setminus V(P)$ and $z\in N_H(\{x_{2t}, x_{2k+1}\})\setminus V(P)$ such that $y\not=z$.
Therefore,
$(x_{2k+2}, x_{2k+3}, \ldots, x_{2t}, z, x_{2k+1}, y, x_0, x_1, \ldots, x_{2k})$ is a linear path of length $t+1$, a contradiction.

(iii) Let $d_P(0,2k+2)=i$ and $d_P(2t,2k)=j$. Then $i,j\ge 1$. Without loss of generality, assume $i\ge j$. If $i+j\ge 5$ then $i\ge 3$. So there must exist $y_1,y_2\in N_H(\{x_0, x_{2k+2}\})\setminus V(P)$ and $z\in N_H(\{x_{2t}, x_{2k}\})\setminus V(P)$ such that $z\notin\{y_1, y_2\}$.
Therefore, $(x_0,x_1,\ldots, x_{2k},z,x_{2t},x_{2t-1}, \ldots, x_{2k+2})\cup \{x_{2k+2}, y_1, x_0\}\cup \{x_{2k+2}, y_2, x_0\}$ is a copy of  $C_{t+1}^{+}$ in $H$, a contradiction.

(iv) Without loss of generality, assume $\ell=0$. Let $d_P(2k,2k+2)=i$ and $d_P(0,2k+1)=j$. Then $i,j\ge 1$. Without loss of generality, assume $i\ge j$. If $i+j\ge 3$ then $i\ge 2$. So we can pick  $y\in N_H(\{x_{2k}, x_{2k+2}\})\setminus V(P)$ and $z\in N_H(\{x_{0}, x_{2k+1}\})\setminus V(P)$ with $y\neq z$.
Therefore, $(x_{2k+1},z,x_0,x_1,\ldots,x_{2k},y,x_{2k+2},x_{2k+3},...,x_{2t})$ is a linear path in $H$ of length $t+1$, a contradiction.
\end{proof}

The following lemma is a corollary of Lemma~\ref{PROP: 4}
\begin{lem}\label{COR: 4}
Given integers $t\ge 3$ and $n\ge 2t+17$, let $H$ be a 3-graph on $n$ vertices and $\delta_1(H)\ge g(n,t)$. If $H$ is $P_{t+1}$-free and $P=(x_0,x_1,\ldots,x_{2t})$ is a linear path of length $t$ in $H$. Then the following statements hold:


(a) $d_P(0, 2k+1)+d_P(2t,2k+1)\le n-2t-1$ for all $k\in [0, t-1]$;


(b) $d_P(0,2k+2)+d_P(2t,2k)\le n-2t-1$ for all $k\in [0, t-1]$.

\end{lem}
\begin{proof}
If one of $d_P(0, h)$ and $d_P(2t, \ell)$ is zero, say $d_P(0,h)=0$, then $d_P(0,h)+d_P(2t,\ell)=0+|N_{H}(\{x_{2t},x_{\ell}\})\setminus V(P)|\le n-2t-1$. So, to prove (a) and (b), it is sufficient to assume  that both $d_P(0,h)$ and $d_P(2t,\ell)$ are positive for $h\in\{2k+1, 2k+2\}$ and $\ell\in\{2k, 2k+1\}$.

(a) Since both $d_P(0,2k+1)$ and $d_P(2t,2k+1)$ are positive, by (ii) of Lemma~\ref{PROP: 4},  we have $d_P(0,2k+1)+d_P(2t,2k+1)\le 2\le n-2t-1$.

(b) Since $d_P(0,2k+2)$ and $d_P(2t,2k)$ are positive, by (iii) of Lemma~\ref{PROP: 4},  we have $d_P(0,2k+2)+d_P(2t,2k)\le 4\le n-2t-1$.
\end{proof}


We first give a weak version of Theorem~\ref{THM: Main1}. For a linear path $P=(x_0,x_1,...,x_{2s})$ in a 3-graph $H$, define $M_P=\{i\in [0, s-1] : d_P(2i,2i+2)\ge 2\}$.

\begin{thm}\label{THM: Weak}
Given integers $t\ge 3$ and $n\ge 2t+17$, every 3-graph $H$ on $n$ vertices with
\begin{equation*}
\delta_{1}(H)\ge\frac{t-1}{2}n+\frac 32t^2-\frac{9}{2}t+6
\end{equation*}
contains a linear path $P_t$ as a subgraph.
\end{thm}
\begin{proof}
Suppose to the contrary that $H$ is $P_{t}$-free. Let $P=(x_0,x_1,...,x_{2s})$ be a longest linear path such that $|M_P|$ has maximum value.
Then $s<t$. Let $T=[0,s-1]\setminus M_P$.

\begin{claim}\label{CLAIM:c2}
For any $X\in N_{H}(x_i)$, $i\in\{0,1, 2s-1,2s\}$, we have $|X\cap V(P)|\ge 1$.
\end{claim}

In fact, if there is an $X\in N_{H}(x_i)$  such that $|X\cap V(P)|=0$ for some $i\in\{0,1, 2s-1,2s\}$, then $P\cup (X\cup\{x_i\})$ is a linear path of length $s+1$, a contradiction.

By Claim~\ref{CLAIM:c2}, we have
\begin{equation}\label{EQ: e1}
d_H(x_0)=\left|N_{H}(x_0)\cap\binom{V(P)}{2}\right|+\sum_{i=1}^{2s}d_P(0,i)\le\binom{2s}{2}+\sum_{i=1}^{2s}d_P(0,i)\mbox{.}
\end{equation}
Similarly, we have $d_H(x_{2s})\le\binom{2s}{2}+\sum\limits_{i=0}^{2s-1}d_P(2s,i)$. So we have
\begin{equation*}
\begin{split}
2\delta_1(H)&\le d_H(x_0)+d_H(x_{2s})\le 2\binom{2s}{2}+\sum_{i=1}^{2s}d_P(0,i)+\sum_{j=0}^{2s-1}d_P(2s,i)\\
 &= 2\binom{2s}{2}+\sum_{k=0}^{s-1}\left[d_P(0,2k+1)+d_P(2s,2k+1)\right]\\
 &\ \  +\sum_{k=0}^{s-1}\left[d_P(0,2k+2)+d_P(2s,2k)\right]\\
            &= 2\binom{2s}{2}+\left(\sum_{k\in M_P}+\sum_{k\in T}\right)\left[d_P(0,2k+1)+d_P(2s,2k+1)\right]\\
            &\ \ +\left(\sum_{k\in M_P}+\sum_{k\in T}\right)\left[d_P(0,2k+2)+d_P(2s,2k)\right]\\
            &\le  2\binom{2s}{2}+0+(s-|M_P|)(n-2s-1)+|M_P|(n-2s-1)\\
            & \ \ \ \ +\sum_{k\in T}\left[d_P(0,2k+2)+d_P(2s,2k)\right]\\
            &= 2\binom{2s}{2}+\sum_{k\in T}\left[d_P(0,2k+2)+d_P(2s,2k)\right]+s(n-2s-1)\mbox{,}
\end{split}
\end{equation*}
where the last inequality holds because $d_P(0,2k+1)=d_P(2s,2k+1)=0$ for $k\in M_P$ by (iv) of Lemma~\ref{PROP: 4},  and $d_P(0,2k+1)+d_P(2s,2k+1)\le n-2s-1$, $d_P(0,2k+2)+d_P(2s,2k)\le n-2s-1$ by Lemma~\ref{COR: 4}.
Therefore,
\begin{equation*}
\begin{split}
& \sum_{k\in T}\left[d_P(0,2k+2)+d_P(2s,2k)\right] \ge 2\delta_1(H)-s(n-2s-1)- 2\binom{2s}{2}\\
                                      &\ \ \ \ \ \ \ \ge (t-1)n+3t^2-9t+12-s(n-2s-1)- 2\binom{2s}{2}\\
                                      &\ \ \ \ \ \ \ \ \ge sn+3s^2-3s+6-s(n-2s-1)-2s(2s-1)\\
                                      &\ \ \ \ \ \ \ \ = s^2+6\mbox{.}
\end{split}
\end{equation*}
This implies that there must exist a $k'\in T$ such that $\max\{d_P(0,2k'+2),d_P(2s,2k')\}$ $\ge \max\{2|M_P|+1,3\}$ (otherwise, if $|M_P|>0$ then $$\sum_{k\in T}\left[d_P(0,2k+2)+d_P(2s,2k)\right]\le\sum_{k\in T}(2|M_P|+2|M_P|)=4|M_P|(s-|M_p|)<s^2+6,$$ a contradiction;
if $|M_P|=0$ then $$\sum_{k\in T}\left[d_P(0,2k+2)+d_P(2s,2k)\right]\le\sum_{k\in T}(2+2)=4s<s^2+6,$$  a contradiction too.)
Without loss of generality, assume $d_P(0,2k'+2)\ge\max\{2|M_P|+1, 3\}$. By pigeon hole principle, we can choose a vertex $v\in N_{H}(\{x_0,x_{2k'+2}\})\setminus V(P)$ such that $v\notin N_H(x_{2k}, x_{2k+2})\setminus V(P)$ for each $k\in M_P$ with $d_P(2k, 2k+2)=2$. Now we set $y_{i}=x_{2k'-i}$ for $i\in [0,2k']$, and $y_{2k'+1}=v_1$ and $y_{j}=x_{j}$ for $j\in [2k+2, 2s]$. Then $P'=(y_0,y_1,...,y_{2s})$ is a linear path of length $s$ in $H$.
Since $|N_H(\{y_{2k'}, y_{2k'+2}\})\setminus V(P')|=|N_H(\{x_{0}, x_{2k'+2}\})\setminus V(P)|-1\ge 2$, we have $k'\in M_{P'}$. For $k> k'$ and $k\in M_P$, if $d_P(2k, 2k+2)\ge 3$ then $d_{P'}(2k, 2k+2)\ge d_P(2k, 2k+2)-1\ge 2$, so $k\in M_{P'}$; if $d_P(2k, 2k+2)=2$ then $d_{P'}(2k, 2k+2)=d_P(2k, 2k+2)=2$, so $k\in M_{P'}$. For $k\le k'-1$ and $k\in M_P$, with a similar discussion with $k>k'$, we have $k'-k-1\in M_{P'}$.
Therefore, we have $|M_{P'}|=|M_P|+1$, a contradiction to the maximality of $M_P$.
\end{proof}

Clearly, (1) of Theorem~\ref{THM: Main1} follows from Theorem~\ref{THM: Weak} by taking $t=2k+1$. So, in the following, we prove the case when $t$ is even.

\begin{thm}\label{THM: Even}
Given positive integers $k$ and $n\ge 4k+21$, every 3-graph $H$ with $\delta_{1}(H)\ge kn+6k^2+7k+6$ contains a linear path $P_{2k+2}$ as a subgraph.
\end{thm}
\begin{proof}
Clearly, $\delta_1(H)\ge kn+6k^2+7k+6\ge kn+6k^2-3k+3$. By (1) of Theorem~\ref{THM: Main1}, $H$ contains a linear path $P_{2k+1}$. Let $P=\{x_0,x_1,...,x_{2\ell}\}$ be a longest linear path in $H$ such that $|M_P|$ has maximum value. If $\ell\ge 2k+2$ then we are done.  Now assume $\ell=2k+1$. 
By (ii) and (iii) of Lemma~\ref{PROP: 4},  $\min\{d_P(0,2i+2),d_P(2\ell,2i)\}\le 2$ for each $i\in M_P$, and  $\min\{d_P(0,2i+1),d_P(2\ell,2i+1)\}\le 1$ for every $i\in [0,\ell-1]\setminus M_P$. Let $T=[0, \ell-1]\setminus M_P$ and
$$N_0=\{i\in M_P: d_P(0,2i+2)\ge 3\}\cup\{i\in T: d_P(0,2i+1)\ge 2\},$$
$$N_{2\ell}=\{i\in M_P: d_P(2\ell,2i)\ge 3\}\cup\{i\in T: d_P(2\ell,2i+1)\ge 2\}.$$
Clearly, we have $N_0\cap N_{2\ell}=\emptyset$. So $|N_0|+|N_{2\ell}|\le \ell$.
Since $\ell=2k+1$, at least one of $|N_0|, |N_{2\ell}|$ is at most $k$. Without loss of generality, assume $|N_0|\le k$.
Let $M_0=[0, \ell-1]\setminus N_0$. Then, similar to Inequality~(\ref{EQ: e1}) in the proof of Theorem~\ref{THM: Weak}, we have
\begin{equation*}
\begin{split}
d_H(x_0)
            &\le\binom{2\ell}{2}+\sum_{i=1}^{2\ell}d_P(0,i)\\
            &=\binom{2\ell}{2}+\left(\sum_{i\in N_0\cap T}+\sum_{i\in M_0\cap T}+\sum_{i\in M_P}\right)d_P(0,2i+1)\\
            &\ \ \ \  + \left(\sum_{i\in N_0\cap M_P}+\sum_{i\in M_0\cap M_P}+\sum_{i\in T}\right)d_P(0,2i+2)\\
            &\le\binom{2\ell}{2}+(n-2\ell-1)|N_0\cap T|+ |M_0\cap T|+0\\
            &+(n-2\ell-1)|N_0\cap M_P|+2|M_0\cap M_P|+\sum_{i\in T}d_P(0,2i+2)\\
            &\le\binom{2\ell}{2}+(n-2\ell-1)|N_0|+2(\ell-|N_0|)+\sum_{i\in T}d_P(0,2i+2)\\
            &\le \ell(2\ell-1)+(n-2\ell-3)k+2\ell+\sum_{i\in T}d_P(0,2i+2)\\
            &\le kn+4k^2+5k+3+\sum_{i\in T}d_P(0,2i+2)\mbox{,}
\end{split}
\end{equation*}
where the second inequality holds because $d_P(0,2i+1)=0$ for any $i\in M_P$ by (iv) of Lemma~\ref{PROP: 4},  and $d_P(0,2i+1)\le n-2s-1$, $d_P(0,2i+2)\le n-2s-1$ by Lemma~\ref{COR: 4}.
Therefore,
\begin{equation*}
\begin{split}
\sum_{i\in T}d_P(0,2i+2)&\ge d_H(x_0)-\left(kn+4k^2+5k+3\right)\\
                        &\ge \delta_1(H)-(kn+4k^2+5k+3)\\
                        & \ge 2k^2+2k+3.
\end{split}
\end{equation*}
This means there exists a $j\in T$ such that $d_P(0,2j+2)\ge \max\{2|M_P|+1,3\}$ (otherwise, if $|M_P|>0$ then $$\sum_{i\in T}d_P(0,2i+2)\le\sum_{i\in T}2|M_P|=2|M_P|(\ell-|M_p|)<2k^2+2k+3,$$ a contradiction;
if $|M_P|=0$ then $$\sum_{i\in T}d_P(0,2i+2)\le\sum_{i\in T}2=2(\ell-|M_P|)<4k+2<2k^2+2k+3,$$  a contradiction too.)
Now with the similar discussion as in the proof of Theorem~\ref{THM: Weak}, we can find a linear path $P'$ of length $2\ell$ with $|K_{P'}|>|K_P|$, which is a contradiction to the choice of $P$.
\end{proof}

Theorem~\ref{THM: Main1} follows directly from Lemma~\ref{LEM: k=0}, Theorems~\ref{THM: Weak} and~\ref{THM: Even}.

\section{Concluding remarks}

In this article, we give an asymptotic upper bound of minimum degree for 3-graphs containing no linear path of specific length. Although the bound is tight up to a constant, we have few information about the extremal 3-graphs through our proofs at this stage. In fact, we believe that $S_3(n,k)$ and $S^+_3(n,k)$ are extremal 3-graphs for $P_{2k+1}$-free and $P_{2k+2}$-free graphs with maximum minimum degree, respectively. We leave this as an open question.

\end{document}